\documentclass[12pt,a4paper]{amsart}
\usepackage{amsfonts}
\usepackage{amsthm}
\usepackage{amsmath}
\usepackage{amscd}
\usepackage[latin2]{inputenc}
\usepackage{t1enc}
\usepackage[mathscr]{eucal}
\usepackage{indentfirst}
\usepackage{graphicx}
\usepackage{graphics}
\usepackage{pict2e}
\usepackage{epic}
\numberwithin{equation}{section}
\usepackage[margin=3.0cm]{geometry}
\usepackage{epstopdf}   
\usepackage{bbm}
\usepackage{hyperref}
\hypersetup{
    colorlinks=true,
    linkcolor=blue,
    filecolor=magenta,      
    urlcolor=cyan,
}

\theoremstyle{plain}
\newtheorem{Th}{Theorem}[section]
\newtheorem{Lemma}[Th]{Lemma}
\newtheorem{Cor}[Th]{Corollary}
\newtheorem{Prop}[Th]{Proposition}

 \theoremstyle{definition}
\newtheorem{Def}[Th]{Definition}

\newtheorem{Rem}[Th]{Remark}
\newtheorem{?}[Th]{Problem}

\DeclareMathOperator{\ber}{Ber}
\begin{document}
\title{On Properties of Random Binary Contingency Tables with Non-Uniform Margin}

\author[Da Wu]{Da Wu}

\address{Department of Mathematics\\ University of Pennsylvania\\ David Rittenhouse Laboratory\\ Philadelphia, PA, 19104-6395} 

\email{dawu@math.upenn.edu}



\begin{abstract} 
In this paper, we study the random binary contingency tables with non-uniform margin. More precisely, for parameters $n,\delta,B,C$, we consider $X=(X_{ij})$ with $ X_{ij}\in \lbrace 0,1\rbrace$, the random binary contingency tables whose first $[n^\delta]$ rows and columns have margin $[BCn]$ and the rest columns and rows have margin $[Cn]$. We study various asymptotic properties of $X$ as $n\to \infty$. This answers a question posted by Barvinok in \cite{4}.   
\end{abstract}
\maketitle
\section{Introduction}
\subsection{Overview}
In Statistics, contingency tables model the dependence structure in large data set. Mathematically, it is simply a set of matrices with fixed row and column sum. Suppose both row sums and column sums depend on the dimension of the matrix and understanding the asymptotic behaviour of contingency tables (as dimension grows) is a challenging task. Combinatorists  are interested in providing a precise asymptotic formula for the cardinality of the contingency tables. In \cite{Canfield 1} , Canfield and Mckay used multi-variable Cauchy Integral Formula to solve the uniform margin case, i.e. all the row sums and column sums are equal (also called \textit{Magic Square} in \cite{R. Stanley}). Later on, Barvinok and Hartigan in \cite{Barvinok and Hartigan} proved a precise asymptotic formula for the non-uniform margin case probabilistically, using \textit{Maximum Entropy Principle}. Their key idea is to express the cardinality of contingency tables as the probability density at a single point and then using the local central limit theorem to complete the approximation. Their formula holds only when the entries of \textit{typical table} remain bounded as dimension goes to infinity. (see later for the precise definition of typical table) The case when some entries of typical matrix blow up remain unsolved. For the readers who are interested in the other combinatorial aspects of contingency tables, see the survey paper \cite{DG95} written by Diaconis and Gangolli. 
\par 
In Probability Theory, we view the set of contingency tables as our ground probability space (whose cardinality is finite) and the probability measure is simply the uniform measure. Again, row sums and column sums depend on the dimension. The fundamental question is that when $n$ is very large, if we uniformly sample one matrix, what does it look like? More precisely, what is the limiting marginal distribution of each entry? What about the joint distribution of sub-matrix? Furthermore, can we say something about the spectrum? In \cite{I.J.Good}, I. J. Good  proposed the so-called \textit{Maximum Entropy Principle} and suggested that it can be applied to the model of Random contingency table. Half a century later, A. Barvinok finally confirmed Good's suggestion, see \cite{2} and the reference therein. Despite the significant progress, the marginal distribution of the entires remained unsolved for large variety of regimes. In \cite{6}, Chatterjee, Diaconis and Sly solved the case of doubly stochastic matrix. More recently, in \cite{1}, Dittmer, Lyu and Pak studied the non-uniform margin case and made important progress on establishing the sharp phase transitions of limiting behaviours. 
\par
In this paper, we study the asymptotic properties of random binary contingency tables with non-uniform margin. Notice that matrices with $0$-$1$ entires and fixed row and column sums is a fundamental object in mathematics. In combinatorics, it has connections with the hypergraphs with fixed degrees of vertices and the network flow, see \cite{Wilson} . It also arises as the structural constants in symmetric function theory hence plays the key role in the representation theory of symmetric and general linear groups, see \cite{I.G.Macdonald}.  
\par
By maximizing the Shannon-Boltzmann entropy of Bernoulli random variables under the first order constraints (fixed row sums and column sums), we obtain the limiting marginal distribution of the uniform sample. Moreover, we show that the joint distribution of entries within each block converges to the i.i.d. Bernoulli, confirming the independent heuristic in \cite{I.J. Good 2} and \cite{I.J. Good 3}. Lastly, we study the convergence (rate) of higher moments and show that strong law of large number holds for certain truncated row.    
\subsection{Basic Setup} 
Let $\mathbf r=(r_1,\ldots,r_m)\in \mathbb N^m$ and $\mathbf c=(c_1,\ldots,c_n)\in \mathbb N^n$ be two positive integer vectors of length $m$ and $n$ with the same sum of entries $N$, i.e., 
\begin{equation*}
	\sum_{i=1}^m r_i=\sum_{j=1}^n c_j=N.
\end{equation*}
We call $\mathbf r$ and $\mathbf c$ row margin and column margin, respectively. Let $\mathscr{M}^{\lbrace 0,1\rbrace}(\mathbf r,\mathbf c)$ be the set of all $m\times n$ binary contingency tables with $i$th row sum $r_i$ and $j$th column sum $c_j$, i.e., 
\begin{equation*}
	\mathscr{M}^{\lbrace 0,1\rbrace}(\mathbf r,\mathbf c):=\left\lbrace (d_{ij})\in \lbrace 0,1\rbrace ^{mn}:\sum_{k=1}^n d_{ik}=r_i,\sum_{k=1}^m d_{kj}=c_j \ \text{for all}\ 1\leq i\leq m,1\leq j\leq n \right\rbrace
\end{equation*}
For $B,C>0$ and $0\leq \delta\leq 1$, let $\widetilde{\mathbf r}=\widetilde{\mathbf c}=\underbrace{([BCn],\ldots,[BCn]}_{[n^\delta]\ \text{entries}},\underbrace{[Cn],\ldots,[Cn])}_{n\ \text{entries} }\in \mathbb N^{[n^{\delta}]+n}$. Define 
\begin{align*}
	\mathscr{M}_{n,\delta}^{\lbrace 0,1\rbrace}(B,C):=\mathscr{M}^{\lbrace 0,1\rbrace}(\widetilde{\mathbf r},\widetilde{\mathbf c}).
\end{align*}
Let $X=(X_{ij})$ be uniformly distributed on $\mathscr{M}_{n,\delta}^{\lbrace 0,1\rbrace}(B,C)$ and we  call $X$ the \textit{Random Binary Contingency Table}. Our goal is to study the limiting distribution of each entry of $\mathscr{M}_{n,\delta}^{\lbrace 0,1\rbrace}(B,C)$ as $n\to \infty$.
\par 
First, we obtain a trivial bound on $B$ and $C$ so that the set $\mathscr{M}_{n,\delta}^{\lbrace 0,1\rbrace}(B,C)$ is always non-empty as $n\to \infty$.  
\begin{Lemma}[Preliminary Analysis]
	As $n\to \infty$, we have the following natural bound on parameter $B$ and $C$:
	\begin{equation*}
		\begin{cases}
			0<C\leq 1,\\
			0<B\leq \frac{1}{C},
		\end{cases}
		\qquad\text{for $0\leq \delta<1$,}
	\end{equation*}
	and 
	\begin{equation*}
		\begin{cases}
			0<C\leq 2,\\
			0<B\leq \frac{2}{C},
		\end{cases}
		\qquad\text{for $\delta=1$}.
	\end{equation*}
\end{Lemma}
\begin{proof}
	Since every entry of the matrix is restricted to $\lbrace 0,1\rbrace$, we have that
	\begin{equation*}
		\begin{cases}
			BCn\leq [n^{\delta}]+n,\\
			Cn\leq [n^{\delta}]+n,
		\end{cases}
	\end{equation*}
	which implies that 
	\begin{equation*}
		\begin{cases}
			BC\leq 1+\frac{[n^{\delta}]}{n},\\
			C\leq 1+\frac{[n^{\delta}]}{n}.
		\end{cases}
	\end{equation*}
	Taking the limit and the results follow.  
\end{proof}
\subsection{Notation}
\begin{enumerate}
	\item For two random variables $X_1,X_2$ taking values on $\mathbb N$, the \textit{total variation distance} between $X_1$ and $X_2$ is defined as
\begin{equation*}
	d_{TV}\left(X_1,X_2\right):=\frac{1}{2} \sum_{k\geq 0}|\mathbb P(X_1=k)-\mathbb P(X_2=k)|.
\end{equation*}
\item Let $\ber(q)$ denote the Bernoulli distribution with mean $q$. Preceisely, if $X\sim\ber(q)$, then $\mathbb P(X=0)=1-q$ and $\mathbb P(X=1)=q$.
\end{enumerate}
\subsection{Main Results}
Our first main result is on the marginal distribution of single entry in $X=(X_{ij})$. One simple observation is that, by symmetry, $X_{ij}$ and $X_{i'j'}$ have the same marginal distribution if $r_i = r_{i'}$ and $c_j = c_{j'}$, where $r_k$ and $c_l$ are the $k$th row sum and $l$th column sum, respectively.
\begin{Th}\label{Main Result}
	For $\mathscr{M}^{\lbrace 0,1\rbrace}_{n,\delta}(B,C)$ with parameter $n,\delta,B,C$, let $X=(X_{ij})$ be sampled uniformly at random from $\mathscr{M}^{\lbrace 0,1\rbrace}_{n,\delta}(B,C)$. Fix $\varepsilon>0$, we have the following:
	\begin{enumerate}
		\item (Bottom Right) For $0\leq \delta<1$, $0<C<1$, and $0<B\leq\frac{1}{C}$, we have that 
	\begin{equation*}
		d_{TV}\left(X_{n+1,n+1},\ber(C)\right)=O\left(n^{\delta-1}+n^{-\frac{1}{2}+\varepsilon} \right).
	\end{equation*}
	\item (Top Left) For $\frac{1}{2}<\delta<1$, $0<C<\frac{3}{4}$, and $B<\frac{1}{\sqrt{\frac{C}{3}-\frac{C^2}{3}}+C}$, we have that 
	\begin{equation*}
		d_{TV}\left(X_{11},\ber\left(\frac{B^2(1-C)}{B^2-2B+1/C}\right)\right)=O\left(n^{\delta-1}+n^{\frac{1}{2}-\delta+\varepsilon} \right).
	\end{equation*}
	\item (Side Blocks) For $0<\delta<1$, $0<C<\frac{3}{4}$, and $B<\frac{1}{\sqrt{\frac{C}{3}-\frac{C^2}{3}}+C}$, we have that 
	\begin{equation*}
		d_{TV}\left(X_{1,n+1}, \ber(BC)\right)=d_{TV}\left(X_{n+1,1}, \ber(BC)\right)=O\left(n^{\delta-1}+n^{-\frac{\delta}{2}+\varepsilon} \right).
	\end{equation*}
	\end{enumerate}
\end{Th}
 It is easy to see that 
\begin{equation}\label{total variation distance between Bernoulli random variable}
	d_{TV}\left(\ber(\lambda_1),\ber(\lambda_2)\right)=2|\lambda_1-\lambda_2|=2\left|\mathbb E[\ber^{k}(\lambda_1)]-\mathbb E[\ber^{k}(\lambda_2)] \right|
\end{equation}
for all $k\geq 1$. This immediately implies the following corollary on the convergence of higher moments:
\begin{Cor}
	For $\mathscr{M}^{\lbrace 0,1\rbrace}_{n,\delta}(B,C)$ with parameter $n,\delta,B,C$, let $X=(X_{ij})$ be uniformly distributed on $\mathscr{M}^{\lbrace 0,1\rbrace}_{n,\delta}(B,C)$. Fix $\varepsilon>0$ and for all $k\geq 1$, we have the following:
	\begin{enumerate}
		\item (Bottom Right) For $0\leq \delta<1$, $0<C\leq 1$, and $0<B\leq\frac{1}{C}$, we have that
	\begin{equation*}
		\left|\mathbb E\left[X^k_{n+1,n+1}\right]-C\right|=O\left(n^{\delta-1}+n^{-\frac{1}{2}+\varepsilon} \right).
	\end{equation*}
	\item (Top Left) For $\frac{1}{2}<\delta<1$, $0<C<\frac{3}{4}$, and $B<\frac{1}{\sqrt{\frac{C}{3}-\frac{C^2}{3}}+C}$, we have that 
	\begin{equation*}
		\left|\mathbb E\left[X^k_{11}\right]-\frac{B^2(1-C)}{B^2-2B+1/C}\right|=O\left(n^{\delta-1}+n^{\frac{1}{2}-\delta+\varepsilon} \right).
	\end{equation*}
	\item (Side Blocks) For $0<\delta<1$, $0<C<\frac{3}{4}$, and $B<\frac{1}{\sqrt{\frac{C}{3}-\frac{C^2}{3}}+C}$, we have that
	\begin{equation*}
		\left|\mathbb E[X^k_{1,n+1}]-BC\right|=O\left(n^{\delta-1}+n^{-\frac{\delta}{2}+\varepsilon} \right).
	\end{equation*}
	\end{enumerate}
\end{Cor}
Our next result deals with the joint distribution of entries within each block. For $k=k(n)$ random variables $R_1,\ldots,R_{k}$, let $(R_1,\ldots, R_k)$ denote the joint distribution of these $k$ random variables. Let $V_k(\gamma)$ denote the joint distribution of $k$ i.i.d. $\ber(\gamma)$. 
\begin{Th}\label{Main results on Joint Distribution}
	For $\mathscr{M}^{\lbrace 0,1\rbrace}_{n,\delta}(B,C)$ with parameter $n,\delta,B,C$, let $X=(X_{ij})$ be uniformly distributed on $\mathscr{M}^{\lbrace 0,1\rbrace}_{n,\delta}(B,C)$. Then we have the following:
	\begin{enumerate}
		\item (Bottom Right) For $0\leq \delta<1$, $0<C<1$, $0<B\leq\frac{1}{C}$, and $k=k(n)=o\left(n^{1-\delta} \right)$, we have that
	\begin{equation*}
		d_{TV}\left((X_{[n^{\delta}]+1,[n^\delta]+1}, X_{[n^\delta]+1,[n^\delta]+2},\ldots, X_{[n^\delta]+1,[n^\delta]+k}),  V_k(C)\right)\to 0.
	\end{equation*}
	\item (Top Left) For $\frac{1}{2}<\delta\leq \frac{2}{3}$, $0<C<\frac{3}{4}$, $B<\frac{1}{\sqrt{\frac{C}{3}-\frac{C^2}{3}}+C}$, and $k=k(n)=o\left(\frac{n^{2\delta-1}}{\log n} \right)$, we have that
	\begin{equation*}
		d_{TV}\left((X_{11}, \ldots, X_{1k}), V_k\left(\frac{B^2(1-C)}{B^2-2B+1/C} \right) \right)\to 0.
	\end{equation*}
	\item (Side Blocks) For $0<\delta\leq\frac{1}{2}$, $0<C<\frac{3}{4}$, $B<\frac{1}{\sqrt{\frac{C}{3}-\frac{C^2}{3}}+C}$, and $k=k(n)=o\left(\frac{n^\delta}{\log n} \right)$, we have that 
	\begin{equation*}
		d_{TV}\left( (X_{1,[n^\delta]+1},\ldots, X_{1,[n^\delta]+k}), V_k(BC)\right)\to 0.
	\end{equation*}
	\end{enumerate}
\end{Th}
The above theorem tells us that, within each block, if we look at the joint distribution of any $k=k(n)$ entries, they are asymptotically independent as $n\to \infty$. In particular, it is true for any fixed number of entries. 
\par
Our final results are on the Law of Large Numbers for certain truncated rows in $X$. Let 
\begin{equation*}
	S^{\textbf{S}}_{n,\delta}(B,C):=\sum_{k=1}^{n}X_{1,k+[n^\delta]}\qquad\text{and}\qquad S_{n,\delta}^{\textbf{BR}}(B,C):=\sum_{k=1}^{n} X_{n+1,k+[n^\delta]}.
\end{equation*}  
\begin{Th}\label{Main results on LLN, CLT, side block}
	For $0<\delta\leq\frac{1}{2}$, $0<C<\frac{3}{4}$, and $B<\frac{1}{\sqrt{\frac{C}{3}-\frac{C^2}{3}}+C}$, let $X=(X_{ij})$ uniformly distributed on $\mathscr{M}^{\lbrace 0,1\rbrace}_{n,\delta}(B,C)$. Then 
	\begin{equation*}
		\frac{1}{n}S^{\textbf{S}}_{n,\delta}(B,C)\to BC
	\end{equation*}
	almost surely.
\end{Th}
\begin{Th}\label{Main results on LLN, CLT, bottom right}
For $0\leq \delta<1$, $0<C<1$, and $0<B\leq\frac{1}{C}$, let $X=(X_{ij})$ be uniformly distributed on $\mathscr{M}^{\lbrace 0,1\rbrace}_{n,\delta}(B,C)$. Then 
	\begin{equation*}
		\frac{1}{n}S^{\textbf{BR}}_{n,\delta}(B,C)\to C
	\end{equation*}
	almost surely.
\end{Th}
\begin{Rem}
	Notice that the expected values of $S^{\textbf{S}}_{n,\delta}(B,C)$ converges to $BCn$, which is already the entire row sum. Therefore, we do not expect the Central Limit Theorem holds for $S^{\textbf{S}}_{n,\delta}(B,C)$ since there is no room for the sum to fluctuate. Same thing goes for $S^{\textbf{BR}}_{n,\delta}(B,C)$. 
\end{Rem}
\subsection{Open Problems}
\begin{enumerate}
	\item In \cite{Neuyen} , Nguyen showed that for the uniformly doubly stochastic matrix, the empirical eigenvalue distribution converges to the circular law. However, for non-uniform margin case, the behaviour of spectrum remains unknown. In fact, we don't even have a conjectural limit.
	\item One simple generalization of our results is to consider the case when the entries of $X$ can take values from $\lbrace 0,1,\ldots, k\rbrace$. The limiting distribution of $X$ is believed to be truncated Geometric by Maximum entropy principle. As $k\to \infty$, it should recover the results in \cite{1}.
	\item Furthermore, It would be more interesting to study the entropy monotonicity (as $k$ and $n$ grow, the entropy increases at each step). This is motivated by the famous \textit{Shannon Monotonicity Conjecture} on Central Limit Theorem. (see \cite{SMC} for the completed proof using Fisher information) I think it is possible to study the higher dimensional version of this conjecture using random contingency table. This is somehow much more general in the sense that not only the dimension increases but also the type of constraints (second order/variance constraints for Central Limit Theorem) changes. The higher dimensional second order/variance constraints correspond to the uniformly/Haar distributed orthogonal/unitary matrix. This is a well-studied family of problems and it is known that the marginal distribution of uniformly distributed Orthogonal/Unitary Matrix converges to standard normal after being rescaled to mean $1$. See, for instance, \cite{Persi 2}, \cite{Borel}, and the reference therein. 
\end{enumerate}
\section{Asymptotic Analysis of Typical Table}
A. Barvinok introduced the notion of \textit{typical table} in order to answer the question \textit{What does a random contingency table look like?} It turns out that as the dimension of matrix grows, the random contingency table is close in certain sense to the typical table. (see \cite{2}, \cite{3}, \cite{4}, \cite{5} for background and the precise statements)
Here we only recall the construction by Barvinok.
\par
Fix margins $\mathbf r\in \mathbb N^m$ and $\mathbf c\in \mathbb N^n$, we define the \textit{binary transportation polytope} to be 
\begin{equation*}
	\mathscr{P}^{\lbrace 0,1\rbrace}(\mathbf r,\mathbf c):=\left\lbrace (x_{ij})\in [0,1]^{mn}: \sum_{k=1}^m x_{ik}=r_i,\sum_{k=1}^n x_{kj}=c_j,\forall 1\leq i\leq m, 1\leq j\leq n \right\rbrace
\end{equation*}
\begin{Def}[Typical table]
	For all $X=(x_{ij})\in (0,1)^{mn}$, let
	\begin{equation*}
		g(X)=\sum_{i,j} x_{ij}\ln\frac{1}{x_{ij}}+(1-x_{ij})\ln\frac{1}{1-x_{ij}}.
	\end{equation*}
	For fixed row and column margins $\mathbf r$ and $\mathbf c$, we define the typical table $Z=(z_{ij})$ to be the unique maximizer of $g$ in the interior of $\mathscr{P}^{\lbrace 0,1\rbrace}(\mathbf r,\mathbf c)$. 
\end{Def}
\begin{Rem}
Notice that 
\begin{enumerate}
	\item Fix $i$ and $j$, we have that the quantity 
	\begin{equation*}
    	x_{ij}\ln\frac{1}{x_{ij}}+(1-x_{ij})\ln\frac{1}{1-x_{ij}}
    \end{equation*}
    is the Shannon-Boltzmann entropy of $\ber(x_{ij})$.
    \item Since $g$ is strictly concave in the interior of $\mathscr{P}^{\lbrace 0,1\rbrace}(\mathbf r,\mathbf c)$, $g$ attains the unique maximum in that region. Therefore, the above definition is well-defined.
    \item Fix $i$ and $j$, we have that
	\begin{equation*}
		\frac{\partial}{\partial x_{ij}}g(X)=\ln\left(\frac{1-x_{ij}}{x_{ij}}\right).
	\end{equation*}
    For the typical table $Z=(z_{ij})$, by Lagrange multiplier condition, we have 
	\begin{equation*}
		\ln\left(\frac{1-z_{ij}}{z_{ij}}\right)=\lambda_i+\mu_j
	\end{equation*}
	for some $\lambda_1,\ldots,\lambda_m$ and $\mu_1,\ldots,\mu_n$.
\end{enumerate}
\end{Rem}
Next, we study the asymptotics of entires $Z=(z_{ij})$. Our analysis in this section follows closely the Lemma $5.1$ and Proposition $5.2$ in Dittmer-Lyu-Pak \cite{1}. The difference is that our optimization is based on the entropy of Bernoulli distribution instead of Geometric distribution. By symmetry and Lagrange multiplier condition, there exists some $\alpha,\beta$ (possibly depend on all the parameters) such that 
\begin{equation*}
	\ln\left(\frac{1-z_{ij}}{z_{ij}}\right)=
	\begin{cases}
		2\alpha &\qquad\text{for $1\leq i,j\leq [n^\delta]$,}\\
		2\beta &\qquad\text{for $[n^\delta]<i,j\leq [n^\delta]+n$,}\\
		\alpha+\beta &\qquad\text{otherwise}.
	\end{cases}
\end{equation*}
Let $P=e^\alpha$ and $Q=e^\beta$, then 
\begin{equation}\label{P,Q change of variables}
	z_{ij}=
	\begin{cases}
		\frac{1}{P^2+1} &\qquad \text{for $1\leq i,j\leq [n^\delta]$,}\\
		\frac{1}{Q^2+1} &\qquad \text{for $[n^\delta]<i,j\leq [n^\delta]+n$,}\\
		\frac{1}{PQ+1} &\qquad\text{otherwise}
	\end{cases}
\end{equation}
By margin conditions of $Z=(z_{ij})$, we have 
\begin{equation}\label{marginal condition}
	\begin{cases}
		([n^\delta]/n)z_{11}+z_{1,n+1}=BC,\\
		([n^\delta]/n)z_{1,n+1}+z_{n+1,n+1}=C.
	\end{cases}
\end{equation}
From $(\ref{marginal condition})$, we get that, 
\begin{align}\label{Asymptotic of top and bottom right}
\begin{cases}
	z_{n+1,n+1}\leq C, \\
	z_{1,n+1}\leq BC,\\
\end{cases}
\qquad\text{and}\qquad
\begin{cases}
	z_{n+1,n+1}=C+O(n^{\delta-1}),\\
	z_{1,n+1}=BC+O(n^{\delta-1}).
\end{cases}
\end{align}
This is because $0<z_{ij}<1$ for all $1\leq i,j\leq n+[n^\delta]$.
\begin{Prop}
	For $0\leq \delta<1$ and $0<C<1$, we have that
	\begin{equation*}
		\limsup_{n\to \infty} z_{11}\leq B^2C< 1.
	\end{equation*}
\end{Prop}
\begin{proof}
	Let $\omega_{11}=z_{11}$, $\omega_{12}=z_{1,n+1}=z_{n+1,1}=\omega_{21}$, and $\omega_{22}=z_{n+1,n+1}$. We have  
	\begin{align*}
		\omega_{11}\omega_{22} =\frac{1}{(Q^2+1)(P^2+1)}=\frac{1}{P^2Q^2+P^2+Q^2+1},
	\end{align*}
	and 
	\begin{align*}
		\omega_{12}\omega_{21} =\frac{1}{(PQ+1)^2}=\frac{1}{P^2Q^2+2PQ+1},
	\end{align*}
	which implies that $\omega_{11}\omega_{22} \leq\omega_{12}\omega_{21}$. Equivalently, we write
	\begin{align}
		\frac{\omega_{11}}{\omega_{12}}\leq \frac{\omega_{21}}{\omega_{22}}.
	\end{align}
	 Next, we claim that $\frac{\omega_{21}}{\omega_{22}}\geq B$. Assume otherwise, i.e., $\frac{\omega_{21}}{\omega_{22}} <B$. Then 
	\begin{equation*}
		\frac{\omega_{11}}{\omega_{12}} \leq \frac{\omega_{21}}{\omega_{22}}=\frac{\omega_{12}}{\omega_{22}}<B,
	\end{equation*}
	and 
	\begin{equation*}
		BC=\frac{[n^\delta]\omega_{11}}{n}+\omega_{12} <\frac{[n^\delta]B\omega_{12} }{n}+B\omega_{22}=BC,
	\end{equation*}
	which is a contradiction. Similarly, we can also show that $\frac{\omega_{11}}{\omega_{12}}\leq B$. 
	Next, notice that 
	\begin{align*}
		\frac{Q}{P}=\frac{PQ}{P^2}=\frac{\omega_{11}(1-\omega_{12})}{\omega_{12}(1-\omega_{11})}=\frac{\frac{1}{\omega_{12}}-1}{\frac{1}{\omega_{11}}-1}.
	\end{align*}
	To find the upper bound for $Q/P$, we solve the following constraint optimization problem:
	\begin{align*}
		&\text{maximize} \ \ \ \frac{\omega_{11}(1-\omega_{12})}{\omega_{12}(1-\omega_{11})},\\
		&\text{subject to} \ \ \ \ \omega_{11}\leq B\omega_{12} \qquad\text{and}\qquad \lim_{n\to \infty} \omega_{12} = BC.
	\end{align*}
	It is easy to see that the objective function is non-decreasing in $\omega_{11}$ and non-increasing in $\omega_{12}$. Hence, 
	\begin{equation*}
		\limsup_{n\to \infty}\frac{Q}{P}\leq \frac{B^2 C(1-BC)}{BC(1-B^2C)}=\frac{B(1-BC)}{1-B^2C}.
	\end{equation*}
	Since $\omega_{12} \leq BC$, we have that $PQ\geq \frac{1}{BC}-1=\frac{1-BC}{BC}$, and 
	\begin{align*}
		\liminf_{n\to\infty} P^2=\liminf_{n\to\infty}\frac{PQ}{Q/P}\geq \frac{\frac{1-BC}{BC}}{\frac{B(1-BC)}{1-B^2C}}=\frac{1-B^2 C}{B^2 C}.
	\end{align*}
	This implies that 
	\begin{align*}
		\limsup_{n\to \infty}z_{11}=\limsup_{n\to \infty}\frac{1}{P^2+1} & \leq B^2 C.
	\end{align*}
\end{proof}
\begin{Lemma}\label{Asymptotic of typical tables}
	Let $Z=(z_{ij})$ be the typical table for $\mathscr{M}_{n,\delta}^{\lbrace 0,1\rbrace}(B,C)$ with $0\leq \delta<1$,
	\begin{align*}
		0<C<\frac{3}{4}, \qquad\text{and}\qquad B<\frac{1}{\sqrt{\frac{C}{3}-\frac{C^2}{3}}+C},
	\end{align*}
	then we have 
	\begin{align*}
		z_{11}=\frac{B^2(1-C)}{B^2-2B+1/C}+O(n^{\delta-1}),\qquad\text{and}\qquad z_{1,n+1}=z_{n+1,1}=BC+O(n^{\delta-1}).
	\end{align*}
\end{Lemma}
\begin{proof}
	Firstly, since $z_{11}$ is uniformly bounded in $n$, 
	\begin{equation}\label{bound on z11}
		|z_{1,n+1}-BC|\leq n^{\delta-1}z_{11}=O(n^{\delta-1}).
	\end{equation}
	This implies $\lim_{n\to \infty}z_{1,n+1}=BC$. Let $P=P(n)$ and $Q=Q(n)$ be as in $(\ref{P,Q change of variables})$, then 
	\begin{equation*}
		\lim_{n\to \infty}z_{1,n+1}=\lim_{n\to \infty}\frac{1}{PQ+1}=BC,\qquad\text{and}\qquad \lim_{n\to \infty}z_{n+1,n+1}=\lim_{n\to \infty}\frac{1}{Q^2+1}=C,
	\end{equation*}
	which is equivalent to 
	\begin{align*}
		Q\to q^*:=\sqrt{\frac{1}{C}-1},\qquad\text{and}\qquad PQ\to \frac{1}{BC}-1.
	\end{align*}
	Consequently,  
	\begin{equation*}
		P\to p^*:=\left(\frac{1}{BC}-1\right)\bigg/\sqrt{\frac{1}{C}-1},
	\end{equation*}
	and 
	\begin{equation*}
		z_{11}=\frac{1}{P^2+1} \to \frac{1}{(p^*)^2+1}=\frac{B^2(1-C)}{B^2-2B+1/C}\leq B^2 C.
	\end{equation*}
	Next, we want to obtain the convergence rate for $z_{11}$. Let $h(x)=\frac{1}{x^2+1}$ and $h'(x)=\frac{-2x}{(x^2+1)^2}$. Since $|h'(x)|$ is decreasing on $(\sqrt{3}/3,\infty)$, when 
	\begin{equation*}
		B<\frac{1}{\sqrt{\frac{C}{3}-\frac{C^2}{3}}+C},
	\end{equation*}
	we have $p^*>\frac{\sqrt{3}}{3}$. By Mean Value Theorem, for all $p$ such that $\sqrt{3}/3<p<p^*$, we have 
	\begin{equation*}
		|h(P)-h(p^*)|\leq |h'(p)||P-p^*|
	\end{equation*}
	for sufficiently large $n$. Therefore, we have
	\begin{align}\label{triangle inequality}
		|P-p^*|\leq \left|P-\frac{1/BC-1}{Q} \right|+\left(\frac{1}{BC}-1\right)\left|\frac{1}{Q}-\frac{1}{q^*} \right|.
	\end{align}
	When $C<\frac{3}{4}$, we have $q^*>\frac{\sqrt{3}}{3}$, and since $z_{n+1,n+1}=h(Q), C=h(q^*)$, the Mean Value Theorem gives us 
	\begin{equation}
		BCn^{\delta-1}\geq |z_{n+1,n+1}-C|=|h(Q)-h(q^*)|\geq |h'(2q^*)|\cdot |Q-q^*|
	\end{equation}
	for sufficiently large $n$. Hence, $|Q-q^*|=O(n^{\delta-1})$. Since $Q\to q^*$, the second term in $(\ref{triangle inequality})$ is of order $O(n^{\delta-1})$. For the first term in $(\ref{triangle inequality})$, we have 
	\begin{align*}
		\left|P-\frac{1/BC-1}{Q}\right| &=\frac{(PQ+1)/BC}{Q}\cdot\left|\frac{1}{PQ+1}-BC \right|\\
		&=\frac{(PQ+1)/BC}{Q}\cdot\left|z_{1,n+1}-BC \right|\\
		&=O(n^{\delta-1}).
	\end{align*} 
	This is because both $P$ and $Q$ converge as $n\to\infty$ and $(\ref{bound on z11})$. Thus $|P-p^*|=O(n^{\delta-1})$, and this completes the proof. 
\end{proof}
\section{Estimation on Total Variation Distance and Proof of Theorem \ref{Main Result} and \ref{Main results on Joint Distribution}}
In this section, we use concentration inequality to prove Theorem \ref{Main Result} and \ref{Main results on Joint Distribution}. The proof is verbatim to that of Theorem 2.1 in \cite{1} and Theorem $1$ in \cite{6}. First, we recall the following theorem by A. Barvinok. 
\begin{Th}$($\cite{4}$)$
	Fix row margin $\mathbf r=(r_1,\ldots, r_m)$ and column margin $\mathbf c=(c_1,\ldots,c_n)$. Let $Z=(z_{ij})_{1\leq i\leq m, 1\leq j\leq n}$ be the typical table for $\mathscr{M}^{\lbrace 0,1\rbrace}(\mathbf r,\mathbf c)$ and let $Y=(y_{ij})_{1\leq i\leq m,1\leq j\leq n}$ be the matrix with independent bernoulli random variables with $y_{ij}\sim \text{Ber}(z_{ij})$. Then we have the following:
	\begin{enumerate}
		\item There exists an absolute constant $\gamma$ such that,
	\begin{align}
		(mn)^{-\gamma(m+n)}e^{g(Z)}\leq \left|\mathscr{M}^{\lbrace 0,1\rbrace}(\mathbf r,\mathbf c)\right|\leq e^{g(Z)}.
	\end{align}
	\item Conditioned on being in $\mathscr{M}^{\lbrace 0,1\rbrace}(\mathbf r,\mathbf c)$, the matrix $Y$ is uniform distributed on $\mathscr{M}^{\lbrace 0,1\rbrace}(\mathbf r,\mathbf c)$. In other words, the probability mass function of $Y$ is constant on the set $\mathscr{M}^{\lbrace 0,1\rbrace}(\mathbf r,\mathbf c)$. More precisely, for any $D\in \mathscr{M}^{\lbrace 0,1\rbrace}(\mathbf r,\mathbf c)$, 
	\begin{align}
		\mathbb P(Y=D)=e^{-g(Z)}.
	\end{align}
	\item There exists some absolute constant $\gamma>0$ such that,  
	\begin{equation}
		\mathbb P\left(Y\in \mathscr{M}^{\lbrace 0,1\rbrace}(\mathbf r,\mathbf c)\right)=e^{-g(Z)}\cdot \left|\mathscr{M}^{\lbrace 0,1\rbrace}(\mathbf r,\mathbf c)\right|\geq (mn)^{-\gamma(m+n)}.
	\end{equation}
	\end{enumerate}
\end{Th}
\begin{Rem}
	For fixed measurable set $\mathscr A\subseteq [0,1]^{mn}$, we have the following transformation of mass inequality: 
	\begin{align}\label{transport inequality}
		\mathbb P(Y\in \mathscr{A})\geq\mathbb P\left(Y\in \mathscr A|Y\in \mathscr{M}^{\lbrace 0,1\rbrace}(\mathbf r,\mathbf c)\right)\cdot \mathbb P\left(Y\in \mathscr{M}^{\lbrace 0,1\rbrace}(\mathbf r,\mathbf c)\right)\geq \mathbb P(X\in \mathscr A)\cdot (mn)^{-\gamma(m+n)}.
	\end{align}
\end{Rem}
Next, we want to obtain an estimate on the total variation distance between entries of $X$ and that of $Y$. 
\begin{Lemma}\label{Bound on total variation distance}
	Let $X=(X_{ij})$ be uniformly distributed on $\mathscr M^{\lbrace 0,1\rbrace}_{n,\delta}(B,C)$ and let $Z=(z_{ij})$ be the typical table for $\mathscr M^{\lbrace 0,1\rbrace}_{n,\delta}(B,C)$. Let $Y=(Y_{ij})$ be the matrix of independent Bernoulli random variables with mean $z_{ij}$, i.e., $Y_{ij}\sim \ber(z_{ij})$. Then, for any fixed $\varepsilon>0$, we have that
		\begin{align}
		\begin{split}
			\begin{cases}
				d_{TV}(X_{11}, Y_{11})=O\left(n^{\frac{1}{2}-\delta+\varepsilon}\right),\\
				d_{TV}(X_{1,n+1}, Y_{1,n+1})=O\left(n^{-\frac{\delta}{2}+\varepsilon}\right),\\
				d_{TV}(X_{n+1,1}, Y_{n+1, 1})=O\left(n^{-\frac{\delta}{2}+\varepsilon}\right),\\
				d_{TV}(X_{n+1,n+1}, Y_{n+1,n+1})=O\left(n^{-\frac{1}{2}+\varepsilon}\right).
			\end{cases}
		\end{split}
	\end{align} 
\end{Lemma}
\begin{proof}
	Fix a measurable set $\mathcal A\subseteq [0,\infty)$. By exchangeability of entries in the top left block and Azuma-Hoeffding inequality, we have that
	\begin{align*}
		\mathbb P\left(\left|\frac{1}{[n^\delta]^2}\sum_{1\leq i\leq [n^\delta]}\sum_{1\leq j\leq [n^\delta]}\mathbbm{1}_{\lbrace Y_{ij}\in \mathcal A\rbrace}-\mathbb P\left(Y_{11}\in \mathcal A\right) \right|>t\right)\leq \exp\left(-2t^2 [n^\delta]^2 \right).
	\end{align*}
	Moreover, by $(\ref{transport inequality})$, 
	\begin{align*}
		&\mathbb P\left(\left|\frac{1}{[n^\delta]^2}\sum_{1\leq i\leq [n^\delta]}\sum_{1\leq j\leq [n^\delta]}\mathbbm{1}_{\lbrace X_{ij}\in \mathcal A\rbrace}-\mathbb P\left(Y_{11}\in \mathcal A\right) \right|>t\right)\\
		&\leq \left(n+[n^\delta]\right)^{\gamma'(n+[n^\delta])}\cdot \exp\left(-2t^2 \left([n^\delta]\right)^2 \right)
	\end{align*}
	for some absolute constant $\gamma'>0$. Next, we have that
	\begin{align*}
		&\left|\mathbb P(X_{11}\in \mathcal A)-\mathbb P(Y_{11}\in \mathcal A) \right|\\
		&=\left|\mathbb E\left[\frac{1}{[n^\delta]^2}\sum_{1\leq i,j\leq [n^\delta]}\mathbbm{1}_{\lbrace X_{ij}\in\mathcal A\rbrace}\right]-\mathbb P(Y_{11}\in \mathcal A) \right|\\
		&\leq \mathbb E\left[\left|\frac{1}{[n^\delta]^2}\sum_{1\leq i,j\leq [n^\delta]}\mathbbm{1}_{\lbrace X_{ij}\in\mathcal A\rbrace}-\mathbb P(Y_{11}\in \mathcal A)\right|\right]\\
		&\leq t\mathbb P\left(\left|\frac{1}{[n^\delta]^2}\sum_{1\leq i,j\leq [n^\delta]}\mathbbm{1}_{\lbrace X_{ij}\in\mathcal A\rbrace}-\mathbb P(Y_{11}\in \mathcal A)\right|\leq t\right)\\
		&+2\mathbb P\left(\left|\frac{1}{[n^\delta]^2}\sum_{1\leq i,j\leq [n^\delta]}\mathbbm{1}_{\lbrace X_{ij}\in \mathcal A\rbrace}-\mathbb P\left(Y_{11}\in \mathcal A\right) \right|>t\right)\\
		&\leq t+2\left(n+[n^\delta]\right)^{\gamma'(n+[n^\delta])}\cdot \exp\left(-2t^2 [n^\delta]^2 \right).
	\end{align*}
	Fix $\varepsilon>0$. Let $t=n^{\frac{1}{2}-\delta+\varepsilon}$, and we have 
	\begin{align}
		\left|\mathbb P(X_{11}\in \mathcal A)-\mathbb P(Y_{11}\in \mathcal A) \right|=O\left(n^{\frac{1}{2}-\delta+\varepsilon}\right).
	\end{align}
	By the exact same method, 
	\begin{align*}
		\left|\mathbb P(X_{1,n+1}\in \mathcal A)-\mathbb P(Y_{1,n+1}\in \mathcal A) \right|\leq t+2\left(n+[n^\delta]\right)^{\gamma{''}(n+[n^\delta])}\cdot \exp\left(-2t^2\cdot [n^\delta]\cdot n \right).
	\end{align*}
	Let $t=n^{-\frac{\delta}{2}+\varepsilon}$ and we have 
	\begin{align}
		\left|\mathbb P(X_{1,n+1}\in \mathcal A)-\mathbb P(Y_{1,n+1}\in \mathcal A) \right|=O\left(n^{-\frac{\delta}{2}+\varepsilon} \right).
	\end{align}
	Finally, 
	\begin{align*}
		\left|\mathbb P(X_{n+1,n+1}\in \mathcal A)-\mathbb P(Y_{n+1,n+1}\in \mathcal A) \right|\leq t+2\left(n+[n^\delta]\right)^{\gamma{'''}(n+[n^\delta])}\cdot \exp\left(-2t^2\cdot n^2  \right).
	\end{align*}
	Let $t=n^{-\frac{1}{2}+\varepsilon}$ and we have 
	\begin{align}
		\left|\mathbb P(X_{n+1,n+1}\in \mathcal A)-\mathbb P(Y_{n+1,n+1}\in \mathcal A) \right|=O\left(n^{-\frac{1}{2}+\varepsilon} \right).
	\end{align}
	Therefore, 
	\begin{align*}
		\begin{split}
			\begin{cases}
				d_{TV}(X_{11}, Y_{11})=O\left(n^{\frac{1}{2}-\delta+\varepsilon}\right),\\
				d_{TV}(X_{1,n+1}, Y_{1,n+1})=O\left(n^{-\frac{\delta}{2}+\varepsilon}\right),\\
				d_{TV}(X_{n+1,1}, Y_{n+1, 1})=O\left(n^{-\frac{\delta}{2}+\varepsilon}\right),\\
				d_{TV}(X_{n+1,n+1}, Y_{n+1,n+1})=O\left(n^{-\frac{1}{2}+\varepsilon}\right).
			\end{cases}
		\end{split}
	\end{align*}
	This completes the proof. 
\end{proof}
Next, we prove the Theorem $\ref{Main Result}$.
\begin{proof}[Proof of Theorem $\ref{Main Result}$]
	By $(\ref{total variation distance between Bernoulli random variable})$, Lemma $\ref{Asymptotic of typical tables}$ and $(\ref{Asymptotic of top and bottom right})$, we have
	\begin{align}\label{Distance between Y and scalar}
	\begin{split}
	\begin{cases}
		d_{TV}\left(\ber(z_{n+1,n+1}),\ber(C)\right)=2|z_{n+1,n+1}-C|=O(n^{\delta-1}),\\
		d_{TV}\left(\ber(z_{1,n+1}),\ber(BC)\right)=2|z_{1,n+1}-BC|=O(n^{\delta-1}),\\
		d_{TV}\left(\ber(z_{n+1,1}),\ber(BC)\right)=2|z_{n+1,1}-BC|=O(n^{\delta-1}),\\
		d_{TV}\left(\ber(z_{11}),\ber\left(\frac{B^2(1-C)}{B^2-2B+1/C} \right) \right)=2\left|z_{11}-\frac{B^2(1-C)}{B^2-2B+1/C}\right|=O(n^{\delta-1}).
	\end{cases}
	\end{split}
	\end{align}
	By Lemma $\ref{Bound on total variation distance}$, we have
	\begin{align}
		\begin{split}
			\begin{cases}
				d_{TV}(X_{11}, Y_{11})=O\left(n^{\frac{1}{2}-\delta+\varepsilon}\right),\\
				d_{TV}(X_{1,n+1}, Y_{1,n+1})=O\left(n^{-\frac{\delta}{2}+\varepsilon}\right),\\
				d_{TV}(X_{n+1,1}, Y_{n+1, 1})=O\left(n^{-\frac{\delta}{2}+\varepsilon}\right),\\
				d_{TV}(X_{n+1,n+1}, Y_{n+1,n+1})=O\left(n^{-\frac{1}{2}+\varepsilon}\right).
			\end{cases}
		\end{split}
	\end{align}
	Hence, by triangle inequality, we have
	\begin{align*}
		d_{TV}(X_{n+1,n+1},\ber(C)) &\leq d_{TV}\left(X_{n+1,n+1},\ber \left(z_{n+1,n+1}\right)\right)+d_{TV}\left(\ber(z_{n+1,n+1}),\ber(C)\right)\\
		&=O\left(n^{\delta-1}+n^{-\frac{1}{2}+\varepsilon} \right),
	\end{align*}
	\begin{align*}
		d_{TV}(X_{1,n+1},\ber(BC)) &\leq d_{TV}\left(X_{1,n+1},\ber \left(z_{1,n+1}\right)\right)+d_{TV}\left(\ber(z_{1,n+1}),\ber(BC)\right)\\
		&=O\left(n^{\delta-1}+n^{-\frac{\delta}{2}+\varepsilon} \right),
	\end{align*}
	\begin{align*}
		d_{TV}(X_{n+1,1},\ber(BC)) &\leq d_{TV}\left(X_{n+1,1},\ber \left(z_{n+1,1}\right)\right)+d_{TV}\left(\ber(z_{n+1,1}),\ber(BC)\right)\\
		&=O\left(n^{\delta-1}+n^{-\frac{\delta}{2}+\varepsilon} \right),
	\end{align*}
	and
	\begin{align*}
		&d_{TV}\left(X_{11},\ber\left(\frac{B^2(1-C)}{B^2-2B+1/C} \right)\right)\\
	    &\leq d_{TV}\left(X_{11},\ber \left(z_{11}\right)\right)+d_{TV}\left(\ber(z_{11}),\ber\left(\frac{B^2(1-C)}{B^2-2B+1/C} \right)\right)\\
		&=O\left(n^{\delta-1}+n^{\frac{1}{2}-\delta+\varepsilon} \right).
	\end{align*}
	This completes the proof. 
\end{proof}
Using the similar techniques, we can study the joint distribution of entries in each block. 
\begin{proof}[Proof of Theorem \ref{Main results on Joint Distribution}]
	We first prove $(1)$. For $k=k(n)$, let $\mathcal A\subseteq \mathbb R^k$ be a measurable subset and let
	\begin{equation*}
		 \mathscr X^{(\ell)}=\lbrace (i^{(\ell)}_1, j^{(\ell)}_1),\ldots,(i^{(\ell)}_k, j^{(\ell)}_k): [n^\delta]+1\leq i_r^{(\ell)}, j_r^{(\ell)}\leq [n^\delta]+n \rbrace
	\end{equation*} 
	 be a $k$-subset of indices of bottom right block. Here $\mathscr X^{(\ell)}\cap \mathscr X^{(\ell')}=\emptyset$ if $1\leq \ell\neq \ell'\leq [n^2/k]$. In other words, we divide the bottom right block into $[n^2/k]$ disjoint subsets, each with cardinality $k$. 
	 \par
	 Let $X^{(\ell)}=(X_{(i^{(\ell)}_1, j^{(\ell)}_1)},\ldots,  X_{(i^{(\ell)}_k, j^{(\ell)}_k)})$ be a random vector of $k$ entries (indexed by $\mathscr X^{(\ell)}$) in the bottom right block of $X$. By symmetry, $X^{(\ell)}$ has the same distribution with $X^{(\ell')}$ for all $1\leq \ell,\ell'\leq [n^2/k]$. Similarly, let $Y^{(\ell)}=(Y_{(i^{(\ell)}_1, j^{(\ell)}_1)},\ldots,  Y_{(i^{(\ell)}_k, j^{(\ell)}_k)})$, where $Y=(Y_{ij})$ is the matrix of independent bernoulli random variables with $Y_{ij}\sim\ber(z_{ij})$. By Azuma-Hoeffding inequality, 
	\begin{align*}
		\mathbb P\left(\left|\frac{1}{[n^2/k]}\sum_{\ell=1}^{[n^2/k]}\mathbbm{1}_{\lbrace X^{(\ell)}\in \mathcal A\rbrace}-\mathbb P\left(Y^{(1)}\in \mathcal A \right)\right|> \frac{1}{2}\varepsilon \right)\leq c'\exp (n\log n)\cdot\exp\left(-\frac{\varepsilon^2}{8}\cdot [n^2/k]^2 \right)
	\end{align*}
	for sufficiently large $n$. Hence, when $k=o\left(\frac{n}{\log n}\right)$, 
	\begin{equation*}
		\mathbb P\left(\left|\frac{1}{[n^2/k]}\sum_{\ell=1}^{[n^2/k]}\mathbbm{1}_{\lbrace X^{(\ell)}\in \mathcal A\rbrace}-\mathbb P\left(Y^{(1)}\in \mathcal A \right)\right|> \frac{1}{2}\varepsilon \right)=o(1).
	\end{equation*}
	Since $X^{(\ell)}$ has the same distribution for all $1\leq \ell\leq [n^2/k]$, we have
	\begin{equation*}
		\left|\mathbb P(X^{(1)}\in \mathcal A)-\mathbb P(Y^{(1)}\in \mathcal A) \right|\leq \frac{1}{2}\varepsilon+o(1).
	\end{equation*}
	Next, 
	\begin{align*}
		d_{TV}\left(Y^{(1)}, V_k(C)\right) &\leq k\cdot d_{TV}\left(\ber(z_{n+1,n+1}),\ber(C) \right)\\
		&=2k|z_{n+1,n+1}-C|
	\end{align*}
	By $(\ref{Distance between Y and scalar})$, $|z_{n+1,n+1}-C|=O(n^{\delta-1})$. Hence, when $k=o\left(n^{1-\delta}\right)$, we have 
	\begin{equation*}
		d_{TV}(Y^{(1)}, V_k(C))=o(1),\qquad\text{and}\qquad d_{TV}(X^{(1)}, Y^{(1)})=o(1)
	\end{equation*}
	By triangle inequality, $(1)$ is proved. Statements $(2)$ and $(3)$ follow from the exact same reasoning, hence details are omitted. 
\end{proof}
A direct consequence of Theorem $\ref{Main results on Joint Distribution}$ is the following corollary. To give the statement, let $J=J(n,\delta,B,C)=(J_{ij})$ be the matrix of independent Bernoulli random variables such that 
\begin{align*}
	J_{ij}\sim 
	\begin{cases}
		\ber(C) &\qquad\text{for $1+[n^\delta]\leq i,j\leq n+[n^\delta]$,}\\
		\ber(BC) &\qquad\text{for $1\leq i\leq [n^\delta]$ and $[n^\delta]+1\leq j\leq [n^\delta]+n$,}\\
		\ber(BC) &\qquad\text{for $1\leq j\leq [n^\delta]$ and $[n^\delta]+1\leq i\leq [n^\delta]+n$,}\\
		\ber\left(\frac{B^2(1-C)}{B^2-2B+1/C}\right) &\qquad\text{for $1\leq i,j\leq [n^\delta]$.}
	\end{cases}
\end{align*}
\begin{Cor}
	Let $(i_1,j_1),\ldots,(i_L,j_L)$ be a fixed sequence of pairs of positive integers and let $\alpha_1,\ldots,\alpha_L$ be a fixed  sequence of positive integers. Under the exact same conditions as Theorem $\ref{Main results on Joint Distribution}$, we have
	\begin{equation}\label{moment convergence}
		\mathbb E\left[\prod_{k=1}^L X_{i_k,j_k}^{\alpha_k}\right]\to \mathbb E\left[ \prod_{k=1}^L J_{i_k,j_k}^{\alpha_k} \right]
	\end{equation}
	if $1\leq i_k,j_k\leq [n^\delta]$ or $[n^\delta]+1\leq i_k,j_k\leq [n^\delta]+n$ or $1\leq i_k\leq [n^\delta]$ and $[n^\delta]+1\leq j_k\leq [n^\delta]+n$ or $1\leq j_k\leq [n^\delta]$ and $[n^\delta]+1\leq i_k\leq [n^\delta]+n$. In other words, $(i_k,j_k)$ are in the same block for all $1\leq k\leq L$. 
\end{Cor}  
\section{Proof of Theorem $\ref{Main results on LLN, CLT, side block}$ and $\ref{Main results on LLN, CLT, bottom right}$}
In this section, we prove the Theorem $\ref{Main results on LLN, CLT, side block}$ and $\ref{Main results on LLN, CLT, bottom right}$. Notice that similar results have been obtained in the non-negative integer case \cite{1}. We first obtain the explicit convergence rate for $(\ref{moment convergence})$. Notice that 
\begin{align}
\begin{split}
	\left|\mathbb E\left[\prod_{k=1}^L X_{i_k,j_k}^{\alpha_k}\right]- \mathbb E\left[ \prod_{k=1}^L J_{i_k,j_k}^{\alpha_k} \right]\right| &=\left|\mathbb P\left(\prod_{k=1}^L X_{i_k,j_k}^{\alpha_k}=1 \right)-\mathbb P\left(\prod_{k=1}^L J_{i_k,j_k}^{\alpha_k}=1 \right) \right|\\
	&=\left|\mathbb P\left(\prod_{k=1}^L X_{i_k,j_k}=1\right)- \mathbb P\left(\prod_{k=1}^L J_{i_k,j_k}=1\right)\right|\\
	&\leq d_{TV}\left(\prod_{k=1}^L X_{i_k,j_k}, \prod_{k=1}^L J_{i_k,j_k}\right)\\
	&\leq d_{TV}\left(\prod_{k=1}^L X_{i_k,j_k}, \prod_{k=1}^L Y_{i_k,j_k}\right)+d_{TV}\left(\prod_{k=1}^L Y_{i_k,j_k}, \prod_{k=1}^L J_{i_k,j_k}\right).
\end{split}
\end{align}
Recall that $Y=(Y_{ij})$ is the matrix of independent Bernoulli random variables with mean $z_{ij}$ and $Z=(z_{ij})$ is the typical table. By symmetry and $(\ref{Distance between Y and scalar})$, we have that 
\begin{equation}
	d_{TV}\left(\prod_{k=1}^L Y_{i_k,j_k}, \prod_{k=1}^L J_{i_k,j_k}\right)\leq L\cdot d_{TV}\left(Y_{i_1,j_1}, J_{i_1,j_1} \right)=O(n^{\delta-1}).
\end{equation}
For the $d_{TV}\left(\prod_{k=1}^L X_{i_k,j_k}, \prod_{k=1}^L Y_{i_k,j_k}\right)$, it can be shown that 
\begin{align}\label{Total variation distance between product}
	d_{TV}\left(\prod_{k=1}^L X_{i_k,j_k}, \prod_{k=1}^L Y_{i_k,j_k}\right)=O\left(n^{-\eta(\delta)+\varepsilon}\right),
\end{align}
where 
\begin{equation*}
	\eta(\delta)=
	\begin{cases}
		\frac{1}{2} &\qquad\text{if all of the $(i_k,j_k)'s$ are in the bottom right block,}\\
		\delta-\frac{1}{2} &\qquad\text{if all of the $(i_k,j_k)'s$ are in the top left block,}\\
		\frac{\delta}{2} &\qquad\text{otherwise.}
	\end{cases}
\end{equation*}
The proof of $(\ref{Total variation distance between product})$ is similar to the proof of Lemma $\ref{Bound on total variation distance}$ above and Theorem $6.1$ in \cite{1} so the details are omitted. Hence, 
\begin{equation}\label{difference between expectaions}
	\left|\mathbb E\left[\prod_{k=1}^L X_{i_k,j_k}^{\alpha_k}\right]- \mathbb E\left[ \prod_{k=1}^L J_{i_k,j_k}^{\alpha_k} \right]\right|=O\left(n^{\delta-1}+n^{-\eta(\delta)+\varepsilon} \right),
\end{equation}
where $\eta(\delta)$ is defined as above. Next, we prove Theorem $\ref{Main results on LLN, CLT, side block}$. The proof of Theorem $\ref{Main results on LLN, CLT, bottom right}$ follows from the exact same reasoning so we will not provide any details here. 
\begin{proof}[Proof of Theorem \ref{Main results on LLN, CLT, side block}]
	Let $\overline X_{1,[n^\delta]+k}=X_{1,[n^\delta]+k}-BC$ for all $1\leq k\leq n$, and let 
	\begin{align*}
		\overline S_{n,\delta}(B,C) &:=X_{1,[n^\delta]+1}+\ldots+X_{1,[n^\delta]+n}-BCn\\
		&=\overline X_{1,[n^\delta]+1}+\ldots+\overline X_{1,[n^\delta]+n}.
	\end{align*}
	By Markov's inequality, we have that
	\begin{align*}
		\mathbb P\left(\overline S_{n,\delta}(B,C)>t\right) &\leq \frac{1}{t^2}\mathbb E\left[\left(\sum_{k=1}^n \overline X_{1,k+[n^\delta]}\right)^2\right]\\
		&=\frac{1}{t^2}\mathbb E\left[\sum_{k=1}^n\overline X^2_{1,k+[n^\delta]}+2\sum_{[n^\delta]+1\leq k_1\neq k_2\leq [n^\delta]+n}\overline X_{1 k_1}\overline X_{1k_2}\right]\\
		&=\frac{1}{t^2}\left\lbrace n\mathbb E\left[\overline X_{1,n+1}^2 \right]+n(n-1)\mathbb E\left[\overline X_{1,n+1}\overline X_{1,n+2} \right] \right\rbrace.
	\end{align*}
	By $(\ref{difference between expectaions})$, we have
	\begin{equation*}
		\mathbb E\left[\overline X_{1,n+1}^2 \right]=BC-B^2C^2+O\left(n^{\delta-1}+n^{-\frac{\delta}{2}+\varepsilon} \right),
	\end{equation*}
	and 
	\begin{equation*}
		\mathbb E\left[\overline X_{1,n+1}\overline X_{1,n+2}\right]=O\left(n^{\delta-1}+n^{-\frac{\delta}{2}+\varepsilon} \right).
	\end{equation*}
	Therefore, when $0<\delta\leq \frac{1}{2}$, we have
	\begin{equation*}
		n\mathbb E\left[\overline X_{1,n+1}^2 \right]+n(n-1)\mathbb E\left[\overline X_{1,n+1}\overline X_{1,n+2} \right]=O\left(n^{2-\frac{\delta}{2}+\varepsilon} \right).
	\end{equation*}
	Hence, for all $\xi,\varepsilon>0$, there exists some constant $c'>0$ such that 
	\begin{equation*}
		\mathbb P\left(\overline S_{n,\delta}(B,C)>n^{1-\xi}\right)\leq c' n^{2\xi-\frac{\delta}{2}+\varepsilon} 
	\end{equation*}
	for sufficiently large $n$. If we choose $0<\xi<\frac{\delta}{4}$, then for some constants $c'',\xi'>0$, we have
	\begin{equation*}
		\mathbb P\left(\frac{\overline S_{n,\delta}(B,C)}{n}>n^{-\xi}\right)\leq c''n^{-\xi'}
	\end{equation*}
	for sufficiently large $n$. For any sequence $(n_k)_{k\geq 1}$ with $n_k\to \infty$ as $k\to \infty$, there exists a subsequence $(n_{k_r})_{r\geq 1}$ with $n_{k_r}\to \infty$ as $r\to \infty$ such that 
	\begin{equation*}
		\sum_{r=1}^\infty \mathbb P\left(\frac{\overline S_{n_{k_r},\delta}(B,C)}{n_{k_r}}>n_{k_r}^{-\xi} \right)<\infty.
	\end{equation*}
	By Borel-Cantelli Lemma,
	\begin{equation*}
		\frac{\overline S_{n_{k_r},\delta}(B,C)}{n_{k_r}}\to 0,	\qquad\text{as $r\to \infty$,}
	\end{equation*} 
	almost surely. Consequently, 
	\begin{equation*}
		\liminf_{n\to \infty}\frac{\overline S_{n,\delta}(B,C)}{n}=\limsup_{n\to \infty}\frac{\overline S_{n,\delta}(B,C)}{n}=0
	\end{equation*}
	almost surely. This completes the proof. 
\end{proof}
\section{Acknowledgement}
The author would like to thank Robin Pemantle for many helpful discussions.

\end{document}